\documentclass{article}

\usepackage{amsmath,amssymb,amsthm}
\usepackage{cite}

\usepackage{a4wide}

\allowdisplaybreaks

\newtheorem{dfn}{Definition}
\newtheorem{thm}[dfn]{Theorem}
\newtheorem{ex}[dfn]{Example}

\newtheorem{lem}[dfn]{Lemma}
\newtheorem{remark}[dfn]{Remark}

\title{Novel results on Hermite--Hadamard kind inequalities\\ 
for $\eta$-convex functions by means\\ 
of $(k,r)$-fractional integral operators\thanks{This is a preprint of a paper 
whose final and definite form is a Springer chapter in the Book
\emph{Advances in Mathematical Inequalities and Applications},
published under the Birkhauser series \emph{Trends in Mathematics},
ISSN: 2297-0215 [{\tt http://www.springer.com/series/4961}].}}

\author{Eze R. Nwaeze$^1$\\
{\tt enwaeze@tuskegee.edu}
\and
Delfim F. M. Torres$^2$\thanks{Corresponding author.}\\
{\tt delfim@ua.pt}}

\date{$^1$Department of Mathematics, Tuskegee University,\\
Tuskegee, AL 36088, USA\\[0.3cm]
$^2$CIDMA, Department of Mathematics, University of Aveiro,\\
3810-193 Aveiro, Portugal}

\begin{document}

\maketitle

\begin{abstract}
We establish new integral inequalities of Hermite--Hadamard type 
for the recent class of $\eta$-convex functions.  
This is done via generalized $(k,r)$-Riemann--Liouville 
fractional integral operators. Our results generalize 
some known theorems in the literature. By choosing different 
values for the parameters $k$ and $r$, 
one obtains interesting new results.

\bigskip


\medskip

\end{abstract}


\section{Introduction}
\label{intro}

Throughout this work, $I\subset\mathbb{R}$ shall denote 
an interval and $I^{\circ}$ the interior of $I$. 
We say that a function $g:I\rightarrow\mathbb{R}$ is convex 
if, for every $x, y\in I$ and $\beta\in[0, 1]$, one has 
\begin{equation}
\label{CD}
g(\beta x+(1-\beta)y)\leq \beta g(x)+(1-\beta)g(y).
\end{equation}
Let $a, b\in I$. For a function $g$ satisfying \eqref{CD}, 
the following inequalities hold:
\begin{equation}
\label{HH1}
g\left(\frac{a+b}{2}\right)
\leq \frac{1}{b-a}\int_a^bg(x)\,dx\leq \frac{g(a)+g(b)}{2}.
\end{equation}
Result \eqref{HH1} was proved by Hadamard in 1893 \cite{Hadamard}
and is celebrated in the literature as the Hermite--Hadamard 
integral inequality for convex functions \cite{MR2761482}. 
Along the years, it has been extended to different classes 
of convex functions: see, e.g., \cite{MR3670501,MR3684864,MR3639749} 
and references therein.

In 2016, the so called $\varphi$-convexity was introduced \cite{GDS},
subsequently denoted as $\eta$-convexity \cite{DS,GDD}. 
Let us recall its definition here.

\begin{dfn}[See \cite{GDS}]
\label{MD}
A function $g:I\rightarrow\mathbb{R}$ is called convex with respect to $\eta$ 
(for short, $\eta$-convex), if 
$$
g(\beta x+ (1-\beta)y)\leq g(y)+\beta \eta(g(x),g(y))
$$ 
for all $x,y\in I$ and $\beta\in[0, 1].$
\end{dfn}

By taking $\eta(x, y)=x-y$, Definition~\ref{MD} reduces to the classical notion 
\eqref{CD} of convexity. It was further shown in \cite{GDS} that for every convex 
function $g$ there exists some $\eta$, different from $\eta(x,y)=x-y$, for which 
the function $g$ is $\eta$-convex. The converse is, however, not necessarily true, 
that is, there are $\eta$-convex functions that are not convex. 

\begin{ex}
Consider function $g:\mathbb{R}\rightarrow\mathbb{R}$ defined piecewisely by
$$
g(x)=
\begin{cases}
-x, ~~x\geq 0,\\
x, ~~x<0,
\end{cases}
$$
and let $\eta:[-\infty,0]\times[-\infty,0]\rightarrow\mathbb{R}$ 
be given by $\eta(x,y)=-x-y$. Function $g$ is clearly not convex 
but it is easy to see that it is $\eta$-convex. Indeed, in 
\cite[Remark~4]{DS} it is noted that an $\eta$-convex function 
$g:[a, b]\rightarrow\mathbb{R}$ is integrable if $\eta$ 
is bounded from above on $g([a, b])\times g([a, b])$.
\end{ex}

For the class of $\eta$-convex functions, the following theorem 
was obtained as an analogue of \eqref{HH1}. 

\begin{thm}[See \cite{GDS}]
Suppose that $g:I\rightarrow\mathbb{R}$ is an $\eta$-convex function 
such that $\eta$ is bounded from above on $g(I)\times g(I)$. 
Then, for any $a,b\in I$ with $a<b$, 
$$
2g\left(\frac{a+b}{2}\right)-M_{\eta}
\leq\frac{1}{b-a}\int_a^b g(x)\,dx
\leq f(b)+\frac{\eta(g(a),g(b))}{2},
$$
where $M_{\eta}$ is an upper bound 
of $\eta$ on $g([a, b])\times g([a, b])$.
\end{thm}

Recently, Rostamian Delavar and De La Sen obtained, among other results, 
the following theorem associated to $\eta$-convex functions \cite{DS}. 

\begin{thm}[See \cite{DS}]
Suppose $g:[a, b]\rightarrow\mathbb{R}$ is a differentiable function and 
$|g'|$ is an $\eta$-convex function with $\eta$ bounded from above 
on $[a, b]$. Then,
\begin{equation*}
\left|\frac{g(a)+g(b)}{2}-\frac{1}{b-a}\int_a^bg(x)\,dx\right|
\leq \frac{1}{8}(b-a)K,
\end{equation*}
where
$K=\min\left\{|g'(b)|+\frac{|\eta(g'(a), g'(b))|}{2},
~ |g'(a)|+\frac{|\eta(g'(b), g'(a))|}{2}\right\}$.
\end{thm}

Still in the same spirit, Khan et al. established in 2017 the following 
result for $\eta$-convex functions via Riemann--Liouville fractional 
integral operators \cite{KKA}. 

\begin{thm}[See \cite{KKA}]
\label{KKAResult} 
Let $g:[a, b]\rightarrow\mathbb{R}$ be a differentiable function on $(a, b)$ with $a<b$. 
If $|g'|$ is an $\eta$-convex function on $[a, b]$, then for $\alpha>0$ the inequality 
\begin{multline*}
\left|\frac{g(a)+g(b)}{2}-\frac{\Gamma(\alpha+1)}{2(b-a)^{\alpha}}
\left[J_{a_+}^{\alpha}g(b) + J_{b_-}^{\alpha}g(a)\right]\right|\\
\leq \frac{b-a}{2(\alpha+1)}\left(1-\frac{1}{2^{\alpha}}\right)
\left(2|g'(b)|+\eta(|g'(a)|, |g'(b)|)\right)
\end{multline*}
holds, where 
$$
J_{a_+}^{\alpha}g(x)=\frac{1}{\Gamma_1(\alpha)}\int_a^x(x-t)^{\alpha-1}g(t)\,dt
$$
is the left Riemann--Liouville fractional integral and
$$
J_{b_-}^{\alpha}g(x)=\frac{1}{\Gamma_1(\alpha)}\int_x^b(t-x)^{\alpha-1}g(t)\,dt
$$
is the right Riemann--Liouville fractional integral.
\end{thm}

Fractional calculus is an area under strong development \cite{MR3614829},
and in \cite{SDKA} Sarikaya et al. proposed the following broader definition 
of the Riemann--Liouville fractional integral operators.

\begin{dfn}[See \cite{SDKA}]
The $(k,r)$-Riemann--Liouville fractional integral operators 
$_{k}^{r}{\mathcal{J}}_{a^+}^{\alpha}$ and $_{k}^{r}{\mathcal{J}}_{b^-}^{\alpha}$ 
of order $\alpha>0$, for a real valued continuous function $g(x)$, are defined as
\begin{equation}
\label{D1}
_{k}^{r}{\mathcal{J}}_{a^+}^{\alpha}g(x)
=\frac{(r+1)^{1-\frac{\alpha}{k}}}{k\Gamma_k(\alpha)}
\int_a^x(x^{r+1}-t^{r+1})^{\frac{\alpha}{k}-1}t^rg(t)\,dt, \quad x>a,
\end{equation}
and  
\begin{equation}
\label{D2}
^{r}_{k}{\mathcal{J}}_{b^-}^{\alpha}g(x)
=\frac{(r+1)^{1-\frac{\alpha}{k}}}{k\Gamma_k(\alpha)}\int_x^b
\left(t^{r+1}-x^{r+1}\right)^{\frac{\alpha}{k}-1}t^rg(t)\,dt, \quad x<b,
\end{equation}
where $k>0$, $r\in\mathbb{R}\setminus\{-1\}$, and $\Gamma_k$ 
is the $k$-gamma function given by 
$$
\Gamma_k(x):=\int_{0}^{\infty}t^{x-1}e^{-\frac{t^k}{k}}\,dt, \quad Re(x)>0,
$$
with the properties $\Gamma_k(x+k)=x\Gamma_k(x)$ and $\Gamma_k(k)=1$.
\end{dfn}

For some results related to the operators \eqref{D1} and \eqref{D2}, 
we refer the interested readers to \cite{Jleli,Mubeen,Set,Tomar}. 
Using these operators, Agarwal et al. established the following 
Hermite--Hadamard type result for convex functions \cite{AMT}.

\begin{thm}[See \cite{AMT}]
\label{AMTresult}
Let $\alpha,\,k>0$ and $r\in\mathbb{R}\setminus\{-1\}$. 
If $g$ is a convex function on $[a, b]$, then
$$
g\left(\frac{a+b}{2}\right)
\leq \frac{(r+1)^{\frac{\alpha}{k}}
\Gamma_k(\alpha+k)}{4(b^{r+1}-a^{r+1})^{\frac{\alpha}{k}}}
\left[_{k}^{r}{\mathcal{J}}_{a^+}^{\alpha}G(b)
+ \,_{k}^{s}{\mathcal{J}}_{b^-}^{\alpha}G(a)\right]
\leq \frac{g(a)+g(b)}{2},
$$
where function $G$ is defined by \eqref{F} below.
\end{thm}

Inspired by the above works, it is our purpose to obtain here 
more general integral inequalities associated to $\eta$-convex functions 
via the $(k,r)$-Riemann--Liouville fractional operators. Theorems~\ref{MR1} 
and \ref{MR2} generalize Theorems~\ref{AMTresult} and \ref{KKAResult}, respectively 
(see Remarks~\ref{r1} and \ref{r2}). In addition, two more fractional
Hermite--Hadamard type inequalities are also established 
(see Theorems~\ref{MR3} and \ref{MR4}).


\section{Main results}
\label{MR}

We establish four new results. For this, 
we start by making the following observations. Let $g$ 
be a function defined on $I$ with $[a, b]\subset I^{\circ}$ 
and define functions
$G,\,\tilde{g}:[a, b]\rightarrow\mathbb{R}$ by 
\begin{equation}
\label{F}
\tilde{g}(x):=g(a+b-x)~~\mbox{and}~~~G(x):=g(x)+\tilde{g}(x).
\end{equation}
For the fractional operators to be well defined, we shall assume 
$g\in L_{\infty}[a, b]$. By making use of the substitutions 
$w=\frac{t-a}{x-a}$ and $w=\frac{b-t}{b-x}$ in \eqref{D1} 
and \eqref{D2}, respectively, one gets that 
\begin{equation}
\label{D3}
_{k}^{r}{\mathcal{J}}_{a^+}^{\alpha}g(x)
=(x-a)\frac{(r+1)^{1-\frac{\alpha}{k}}}{k\Gamma_k(\alpha)}
\int_0^1\frac{(wx+(1-w)a)^rg(wx+(1-w)a)}{\left[
x^{r+1}-(wx+(1-w)a)^{r+1}\right]^{1-\frac{\alpha}{k}}}\,dw
\end{equation}
and
\begin{equation}
\label{D4}
_{k}^{r}{\mathcal{J}}_{b^-}^{\alpha}g(x)
=(b-x)\frac{(r+1)^{1-\frac{\alpha}{k}}}{k\Gamma_k(\alpha)}
\int_0^1\frac{(wx+(1-w)b)^rg(wx+(1-w)b)}{\left[
(wx+(1-w)b)^{r+1}-x^{r+1}\right]^{1-\frac{\alpha}{k}}}\,dw.
\end{equation}
Noting that $\tilde{g}\left((1-w)a+wb\right)=g\left(wa+(1-w)b\right)$, 
we also obtain
\begin{equation}
\label{D5}
_{k}^{r}{\mathcal{J}}_{a^+}^{\alpha}\tilde{g}(x)
=(x-a)\frac{(r+1)^{1-\frac{\alpha}{k}}}{k\Gamma_k(\alpha)}
\int_0^1\frac{(wx+(1-w)a)^rg((1-w)x+wa)}{\left[
x^{r+1}-(wx+(1-w)a)^{r+1}\right]^{1-\frac{\alpha}{k}}}\,dw
\end{equation}
and 
\begin{equation}
\label{D6}
_{k}^{r}{\mathcal{J}}_{b^-}^{\alpha}\tilde{g}(x)
=(b-x)\frac{(r+1)^{1-\frac{\alpha}{k}}}{k\Gamma_k(\alpha)}
\int_0^1\frac{(wx+(1-w)b)^rg((1-w)x+wb)}{\left[
(wx+(1-w)b)^{r+1}-x^{r+1}\right]^{1-\frac{\alpha}{k}}}\,dw.
\end{equation}
We are now ready to formulate and prove our first result.

\begin{thm}
\label{MR1}
Let $\alpha,\,k>0$, $r\in\mathbb{R}\setminus\{-1\}$, and $g:I\rightarrow\mathbb{R}$ 
be a positive function on $[a, b]\subset I^{\circ}$ with $a<b$.  If, in addition, 
$g$ is $\eta$-convex on $[a, b]$ with $\eta$ bounded on $g([a,b])\times g([a,b])$, 
then the $(k, r)$-fractional integral inequality 
\begin{align*}
\frac{(r+1)^{\frac{\alpha}{k}}\Gamma_k(\alpha+k)}{4(b^{r+1}-a^{r+1})^{\frac{\alpha}{k}}}
\left[_{k}^{r}{\mathcal{J}}_{a^+}^{\alpha}G(b)
+ \,_{k}^{r}{\mathcal{J}}_{b^-}^{\alpha}G(a)\right]
\leq g(b)+\frac{\eta(g(a),g(b))}{2}
\end{align*}
holds.
\end{thm}

\begin{proof}
Function $g$ is $\eta$-convex on $[a, b]$, which implies, 
by definition, the following inequalities for $t\in[0,1]$:
\begin{equation}
\label{Q1}
g(ta+(1-t)b)\leq g(b)+t\eta(g(a),g(b))
\end{equation}
and 
\begin{equation}
\label{Q2}
g((1-t)a+tb)\leq g(b)+(1-t)\eta(g(a),g(b)).
\end{equation}
Adding inequalities \eqref{Q1} and \eqref{Q2}, we get
\begin{equation}
\label{Q3}
g(ta+(1-t)b)+ g((1-t)a+tb)\leq 2g(b)+\eta(g(a),g(b)).
\end{equation}
Multiplying both sides of \eqref{Q3} by 
$$
(b-a)\frac{(r+1)^{1-\frac{\alpha}{k}}}{k\Gamma_k(\alpha)}
\frac{(tb+(1-t)a)^r}{\left[b^{r+1}
-(tb+(1-t)a)^{r+1}\right]^{1-\frac{\alpha}{k}}},
$$
and integrating over $[0, 1]$ with respect to $t$, we get
\begin{align*}
&(b-a)\frac{(r+1)^{1-\frac{\alpha}{k}}}{k\Gamma_k(\alpha)}
\int_0^1\frac{(tb+(1-t)a)^rg((1-t)b+ta)}{\left[
b^{r+1}-(tb+(1-t)a)^{r+1}\right]^{1-\frac{\alpha}{k}}}\,dt\\
&+(b-a)\frac{(r+1)^{1-\frac{\alpha}{k}}}{k\Gamma_k(\alpha)}\int_0^1
\frac{(tb+(1-t)a)^rg(tb+(1-t)a)}{\left[
b^{r+1}-(tb+(1-t)a)^{r+1}\right]^{1-\frac{\alpha}{k}}}\,dt\\
&\leq \left[2g(b)+\eta(g(a),g(b))\right](b-a)
\frac{(r+1)^{1-\frac{\alpha}{k}}}{k\Gamma_k(\alpha)}
\int_0^1\frac{(tb+(1-t)a)^r}{\left[
b^{r+1}-(tb+(1-t)a)^{r+1}\right]^{1-\frac{\alpha}{k}}}\,dt.
\end{align*}
Now, using \eqref{D3} and \eqref{D5} in the above inequality, we get
$$
_{k}^{r}{\mathcal{J}}_{a^+}^{\alpha}\tilde{g}(b)
+ _{k}^{r}{\mathcal{J}}_{a^+}^{\alpha}g(b)
\leq\frac{(s+1)^{1-\frac{\alpha}{k}}(b^{s+1}-a^{s+1})^{\frac{\alpha}{k}}}{(s+1)
\alpha\Gamma_{k}(\alpha)}\left[2g(b)+\eta(g(a),g(b))\right],
$$
that is,
\begin{equation}
\label{eq1}
_{k}^{r}{\mathcal{J}}_{a^+}^{\alpha}G(b)
\leq \frac{(b^{r+1}-a^{r+1})^{\frac{\alpha}{k}}}{(r+1)^{\frac{\alpha}{k}}
\Gamma_{k}(\alpha+k)}\left[2g(b)+\eta(g(a),g(b))\right].
\end{equation}
Similarly, multiplying again both sides of \eqref{Q3} by 
$$
(b-a)\frac{(r+1)^{1-\frac{\alpha}{k}}}{k\Gamma_k(\alpha)}
\frac{(tb+(1-t)a)^r}{\left[(tb+(1-t)a)^{r+1}-a^{r+1}\right]^{1-\frac{\alpha}{k}}}
$$
and integrating with respect to $t$ over $[0, 1]$, we obtain that
\begin{equation}
\label{eq2}
_{k}^{r}{\mathcal{J}}_{b^-}^{\alpha}G(a)
\leq \frac{(b^{r+1}-a^{r+1})^{\frac{\alpha}{k}}}{(r+1)^{\frac{\alpha}{k}}
\Gamma_{k}(\alpha+k)}\left[2g(b)+\eta(g(a),g(b))\right].
\end{equation}
Hence, the intended inequality follows by adding \eqref{eq1} and \eqref{eq2}.
\end{proof}

\begin{remark}
\label{r1}
By taking $\eta(x,y)=x-y$ in our Theorem~\ref{MR1}, 
we recover the right-hand side of the inequalities 
in Theorem~\ref{AMTresult}.
\end{remark}

For the rest of our results, we will need the following two lemmas.

\begin{lem}[See \cite{AMT}]
\label{lem1}
Let $\alpha,\,k>0$ and $r\in\mathbb{R}\setminus\{-1\}$. 
If $g:I\rightarrow\mathbb{R}$ is differentiable on $I^{\circ}$ 
and $a, b\in I^{\circ}$ such that $g'\in L[a, b]$ with $a<b$, 
then the following identity holds:
\begin{multline*}
\frac{g(a)+g(b)}{2}-\frac{(r+1)^{\frac{\alpha}{k}}
\Gamma_k(\alpha+k)}{4(b^{r+1}-a^{r+1})^{\frac{\alpha}{k}}}
\left[_{k}^{r}{\mathcal{J}}_{a^+}^{\alpha}G(b)
+ \,_{k}^{s}{\mathcal{J}}_{b^-}^{\alpha}G(a)\right]\\
=\frac{b-a}{4(b^{r+1}-a^{r+1})^{\frac{\alpha}{k}}}
\int_0^1\Theta_{\alpha,r}(t)g'(ta+(1-t)b)\,dt,
\end{multline*}
where
$\Theta_{\alpha,r}:[0, 1]\rightarrow\mathbb{R}$ is defined by
\begin{multline*}
\Theta_{\alpha,r}(t)
:=\left[(ta+(1-t)b)^{r+1}-a^{r+1}\right]^{\frac{\alpha}{k}}
-\left[(tb+(1-t)a)^{r+1}-a^{r+1}\right]^{\frac{\alpha}{k}}\\
+ \left[b^{r+1}-(tb+(1-t)a)^{r+1}\right]^{\frac{\alpha}{k}}
-\left[b^{r+1}-(ta+(1-t)b)^{r+1}\right]^{\frac{\alpha}{k}}.
\end{multline*}
\end{lem}

\begin{lem}
\label{lem2}
Under the conditions of Lemma~\ref{lem1}, we have that 
$$
\int_0^1|\Theta_{\alpha, r}(t)|\,dt
=\frac{1}{b-a}\left(\Re_1+\Re_2+\Re_3+\Re_4\right),
$$
where
$$
\Re_1=\int^b_{\frac{a+b}{2}}\left(w^{r+1}-a^{r+1}\right)^{\frac{\alpha}{k}}\,dw
-\int^{\frac{a+b}{2}}_a\left(w^{r+1}-a^{r+1}\right)^{\frac{\alpha}{k}}\,dw,
$$
$$
\Re_2=\int^b_{\frac{a+b}{2}}\left[b^{r+1}-(b+a-w)^{r+1}\right]^{\frac{\alpha}{k}}\,dw
-\int^{\frac{a+b}{2}}_a\left[b^{r+1}-(b+a-w)^{r+1}\right]^{\frac{\alpha}{k}}\,dw,
$$
$$
\Re_3=\int_a^{\frac{a+b}{2}}\left(b^{r+1}-w^{r+1}\right)^{\frac{\alpha}{k}}\,dw
-\int_{\frac{a+b}{2}}^b\left(b^{r+1}-w^{r+1}\right)^{\frac{\alpha}{k}}\,dw,
$$
and
$$
\Re_4=\int_a^{\frac{a+b}{2}}\left[(b+a-w)^{r+1}-a^{r+1}\right]^{\frac{\alpha}{k}}\,dw
-\int_{\frac{a+b}{2}}^b\left[(b+a-w)^{r+1}-a^{r+1}\right]^{\frac{\alpha}{k}}\,dw.
$$
\end{lem}

\begin{proof}
Using the substitution $w=ta+(1-t)b$, we get
\begin{equation}
\label{p}
\int_0^1|\Theta_{\alpha, r}(t)|\,dt=\frac{1}{b-a}\int_a^b|\wp(w)|\,dw,
\end{equation}
where 
\begin{multline*}
\wp(w)=\left(w^{r+1}-a^{r+1}\right)^{\frac{\alpha}{k}}
-\left[(b+a-w)^{r+1}-a^{r+1}\right]^{\frac{\alpha}{k}}\\
+ \left[b^{r+1}-(b+a-w)^{r+1}\right]^{\frac{\alpha}{k}}
-\left(b^{r+1}-w^{r+1}\right)^{\frac{\alpha}{k}}.
\end{multline*}
The required result follows from \eqref{p} and by observing 
that $\wp$ is a non-decreasing function on $[a, b]$, 
$\wp(a)=-2(b^{r+1}-a^{r+1})^{\frac{\alpha}{k}}<0$, 
$\wp\left(\frac{a+b}{2}\right)=0$, and thus
\begin{equation*}
\begin{cases}
\wp(w)\leq 0 & \mbox{if}\quad a\leq w\leq\frac{a+b}{2},\\
\wp(w)>0 & \mbox{if}\quad \frac{a+b}{2}< w\leq b.
\end{cases}
\end{equation*}
This concludes the proof.
\end{proof}

\begin{thm}
\label{MR2}
Let $\alpha,\,k>0$, $r\in\mathbb{R}\setminus\{-1\}$, $g:I\rightarrow\mathbb{R}$  
be a differentiable function on $I^{\circ}$ and $a, b\in I^{\circ}$ with $a<b$.  
Suppose $|g'|$ is $\eta$-convex on $[a, b]$ with $\eta$ bounded on 
$|g'|([a,b])\times |g'|([a,b])$. Then the following
$(k, r)$-fractional integral inequality holds:
\begin{multline*}
\left|\frac{g(a)+g(b)}{2}-\frac{(r+1)^{\frac{\alpha}{k}}
\Gamma_k(\alpha+k)}{4(b^{r+1}-a^{r+1})^{\frac{\alpha}{k}}}
\left[_{k}^{r}{\mathcal{J}}_{a^+}^{\alpha}G(b)
+ \,_{k}^{r}{\mathcal{J}}_{b^-}^{\alpha}G(a) \right]\right|\\
\leq \frac{1}{4(b^{r+1}-a^{r+1})^{\frac{\alpha}{k}}}
\left[\Re|g'(b)|+\frac{\Xi}{b-a}\eta(|g'(a)|,|g'(b)|)\right],
\end{multline*}
where
$\Re=\Re_1+\Re_2+\Re_3+\Re_4$ (see Lemma~\ref{lem2}) 
and $\Xi=\xi_1+\xi_2+\xi_3+\xi_4$ with
$$
\xi_1=\int_a^{\frac{a+b}{2}}(b-w)(b^{r+1}-w^{r+1})^{\frac{\alpha}{k}}\,dw
-\int_{\frac{a+b}{2}}^b(b-w)(b^{r+1}-w^{r+1})^{\frac{\alpha}{k}}\,dw,
$$
$$
\xi_2=\int_{\frac{a+b}{2}}^b(b-w)(w^{r+1}-a^{r+1})^{\frac{\alpha}{k}}\,dw
-\int_a^{\frac{a+b}{2}}(b-w)(w^{r+1}-a^{r+1})^{\frac{\alpha}{k}}\,dw,
$$
$$
\xi_3=\int_a^{\frac{a+b}{2}}(b-w)((b+a-w)^{r+1}-a^{r+1})^{\frac{\alpha}{k}}\,dw
-\int_{\frac{a+b}{2}}^b(b-w)((b+a-w)^{r+1}-a^{r+1})^{\frac{\alpha}{k}}\,dw,
$$
$$
\xi_4=\int_{\frac{a+b}{2}}^b(b-w)(b^{r+1}-(b+a-w)^{r+1})^{\frac{\alpha}{k}}\,dw
-\int_a^{\frac{a+b}{2}}(b-w)(b^{r+1}-(b+a-w)^{r+1})^{\frac{\alpha}{k}}\,dw.
$$
\end{thm}

\begin{proof}
Since $|f'|$ is $\eta$-convex, it follows, by definition, that 
\begin{equation}
\label{Qua}
\left|g'(ta+(1-t)b)\right|
\leq |g'(b)|+t\eta\left(|g'(a)|,|g'(b)|\right)
\end{equation}
for $t\in[0,1]$. From \cite[p.~9]{AMT}, we have
\begin{equation}
\label{id}
\int_0^1t|\Theta_{\alpha, r}(t)|\,dt=\frac{\xi_1+\xi_2+\xi_3+\xi_4}{(b-a)^2}.
\end{equation}
Using Lemmas~\ref{lem1} and \ref{lem2}, inequality \eqref{Qua}, 
identity \eqref{id}, and properties of the modulus, we obtain
\begin{align*}
&\left|\frac{g(a)+g(b)}{2}-\frac{(r+1)^{\frac{\alpha}{k}}
\Gamma_k(\alpha+k)}{4(b^{r+1}-a^{r+1})^{\frac{\alpha}{k}}}
\left[_{k}^{r}{\mathcal{J}}_{a^+}^{\alpha}G(b)
+ \,_{k}^{r}{\mathcal{J}}_{b^-}^{\alpha}G(a) \right]\right|\\
&\leq \frac{b-a}{4(b^{r+1}-a^{r+1})^{\frac{\alpha}{k}}}
\int_0^1|\Theta_{\alpha,r}(t)||g'(ta+(1-t)b)|\,dt\\
&\leq \frac{b-a}{4(b^{r+1}-a^{r+1})^{\frac{\alpha}{k}}}
\int_0^1|\Theta_{\alpha,r}(t)|\left[|g'(b)|+t\eta(|g'(a)|,|g'(b)|)\right]\,dt\\
&=\frac{b-a}{4(b^{r+1}-a^{r+1})^{\frac{\alpha}{k}}}\left(|g'(b)|
\int_0^1|\Theta_{\alpha,r}(t)|\,dt+\eta(|g'(a)|,|g'(b)|)
\int_0^1t|\Theta_{\alpha,r}(t)|\,dt\right)\\
&=\frac{b-a}{4(b^{s+1}-a^{s+1})^{\frac{\alpha}{k}}}
\left[|g'(b)|\frac{1}{b-a}\left(\Re_1+\Re_2+\Re_3+\Re_4\right)
+ \eta(|g'(a)|,|g'(b)|)\frac{\xi_1+\xi_2+\xi_3+\xi_4}{(b-a)^2}\right].
\end{align*}
The desired result follows.
\end{proof}

\begin{remark}
\label{r2}
By taking $r=0$ and $k=1$ in Theorem~\ref{MR2}, we recover Theorem~\ref{KKAResult}. 
In this case, 
$$
\Re=\frac{4}{\alpha+1}(b-a)^{\alpha+1}\left(1-\frac{1}{2^{\alpha}}\right)
$$
and
$$
\Xi=\frac{2}{\alpha+1}(b-a)^{\alpha+2}\left(1-\frac{1}{2^{\alpha}}\right).
$$
\end{remark}

\begin{thm}
\label{MR3}
Let $g$ be differentiable on $I^{\circ}$ with $a, b\in I^{\circ}$. 
If $|g'|^q$ is $\eta$-convex on $[a,b]$ and $q>1$ with $\eta$ bounded 
on $|g'|^q([a,b])\times |g'|^q([a,b])$, then the 
$(k, r)$-fractional integral inequality 
\begin{multline*}
\left|\frac{g(a)+g(b)}{2}-\frac{(r+1)^{\frac{\alpha}{k}}
\Gamma_k(\alpha+k)}{4(b^{r+1}-a^{r+1})^{\frac{\alpha}{k}}}
\left[_{k}^{r}{\mathcal{J}}_{a^+}^{\alpha}G(b)
+ \,_{k}^{r}{\mathcal{J}}_{b^-}^{\alpha}G(a) \right]\right|\\
\leq \frac{b-a}{4(b^{r+1}-a^{r+1})^{\frac{\alpha}{k}}}
\left( |g'(b)|^q+\frac{\eta(|g'(a)|^q,|g'(b)|^q)}{2}
\right)^{\frac{1}{q}}||\Theta_{\alpha,r}||_p
\end{multline*}
holds, where $\frac{1}{p}+\frac{1}{q}=1$ and 
$||\Theta_{\alpha,r}||_p
=\left(\int_0^1|\Theta_{\alpha,r}(t)|^p\,dt\right)^{\frac{1}{p}}$.
\end{thm}

\begin{proof}
Function $|g'|^q$ is $\eta$-convex, which implies 
\begin{equation}
\label{Quaq}
|g'(ta+(1-t)b)|^q\leq |g'(b)|^q+t\eta(|g'(a)|^q,|g'(b)|^q),
\end{equation}
$t\in[0,1]$. Using Lemma~\ref{lem1}, inequality \eqref{Quaq}, 
H\"older's inequality and the properties of modulus, we get
\begin{align*}
&\left|\frac{g(a)+g(b)}{2}-\frac{(r+1)^{\frac{\alpha}{k}}
\Gamma_k(\alpha+k)}{4(b^{r+1}-a^{r+1})^{\frac{\alpha}{k}}}
\left[_{k}^{r}{\mathcal{J}}_{a^+}^{\alpha}G(b)
+ \,_{k}^{r}{\mathcal{J}}_{b^-}^{\alpha}G(a) \right]\right|\\
&\leq \frac{b-a}{4(b^{r+1}-a^{r+1})^{\frac{\alpha}{k}}}
\int_0^1|\Theta_{\alpha,r}(t)||g'(ta+(1-t)b)|\,dt\\
&\leq\frac{b-a}{4(b^{r+1}-a^{r+1})^{\frac{\alpha}{k}}}\left(
\int_0^1|\Theta_{\alpha,r}(t)|^p\,dt\right)^{\frac{1}{p}}
\left(\int_0^1 |g'(ta+(1-t)b)|^q\,dt\right)^{\frac{1}{q}}\\
&\leq\frac{b-a}{4(b^{r+1}-a^{r+1})^{\frac{\alpha}{k}}}\left(
\int_0^1|\Theta_{\alpha,r}(t)|^p\,dt\right)^{\frac{1}{p}}\left(
\int_0^1\left[ |g'(b)|^q+t\eta(|g'(a)|^q,|g'(b)|^q)\right]\,dt\right)^{\frac{1}{q}}\\
&=\frac{b-a}{4(b^{r+1}-a^{r+1})^{\frac{\alpha}{k}}}\left(\int_0^1
|\Theta_{\alpha,r}(t)|^p\,dt\right)^{\frac{1}{p}}\left( |g'(b)|^q
+\frac{\eta(|g'(a)|^q,|g'(b)|^q)}{2}\right)^{\frac{1}{q}}.
\end{align*}
This completes the proof.
\end{proof}

\begin{thm}
\label{MR4}
Let $g$ be differentiable on $I^{\circ}$ with $a, b\in I^{\circ}$. 
If $|g'|^q$ is $\eta$-convex on $[a,b]$ and $q>1$ with $\eta$ bounded 
on $|g'|^q([a,b])\times |g'|^q([a,b])$, then the 
$(k, r)$-fractional integral inequality 
\begin{multline*}
\left|\frac{g(a)+g(b)}{2}-\frac{(r+1)^{\frac{\alpha}{k}}
\Gamma_k(\alpha+k)}{4(b^{r+1}-a^{r+1})^{\frac{\alpha}{k}}}
\left[_{k}^{r}{\mathcal{J}}_{a^+}^{\alpha}G(b)
+ \,_{k}^{r}{\mathcal{J}}_{b^-}^{\alpha}G(a)\right]\right|\\
\leq\frac{\Re^{\frac{1}{p}}}{4(b^{r+1}-a^{r+1})^{\frac{\alpha}{k}}}
\left[\left|g'(b)\right|^q+\frac{\Xi}{b-a}\eta\left(
|g'(a)|^q,|g'(b)|^q\right)\right]^{\frac{1}{q}}
\end{multline*}
holds, where $\frac{1}{p}+\frac{1}{q}=1$ and $\Re$ and $\Xi$ 
are defined as in Theorem~\ref{MR2}.
\end{thm}

\begin{proof}
Following a similar approach as in the proof of Theorem~\ref{MR3}, 
we have, by using  Lemmas~\ref{lem1} and \ref{lem2} combined with 
the power mean inequality plus inequality \eqref{Quaq}, that
\begin{align*}
&\left|\frac{g(a)+g(b)}{2}-\frac{(r+1)^{\frac{\alpha}{k}}
\Gamma_k(\alpha+k)}{4(b^{r+1}-a^{r+1})^{\frac{\alpha}{k}}}
\left[_{k}^{r}{\mathcal{J}}_{a^+}^{\alpha}G(b)
+ \,_{k}^{r}{\mathcal{J}}_{b^-}^{\alpha}G(a)\right]\right|\\
&\leq \frac{b-a}{4(b^{r+1}-a^{r+1})^{\frac{\alpha}{k}}}
\int_0^1|\Theta_{\alpha,r}(t)||g'(ta+(1-t)b)|\,dt\\
&\leq\frac{b-a}{4(b^{r+1}-a^{r+1})^{\frac{\alpha}{k}}}\left(
\int_0^1|\Theta_{\alpha,r}(t)|\,dt\right)^{1-\frac{1}{q}}\left(
\int_0^1|\Theta_{\alpha,r}(t)| |g'(ta+(1-t)b)|^q\,dt\right)^{\frac{1}{q}}\\
&\leq\frac{b-a}{4(b^{r+1}-a^{r+1})^{\frac{\alpha}{k}}}\left(
\int_0^1|\Theta_{\alpha,r}(t)|\,dt\right)^{1-\frac{1}{q}}\left(
\int_0^1|\Theta_{\alpha,r}(t)|\left[ |g'(b)|^q
+t\eta(|g'(a)|^q,|g'(b)|^q)\right]\,dt\right)^{\frac{1}{q}}.
\end{align*}
The required inequality follows.
\end{proof}


\section*{Acknowledgements}

This research was supported by FCT and CIDMA,
project UID/MAT/04106/2013. The authors are grateful
to the referees for their valuable comments 
and helpful suggestions.



\end{document}